\documentclass[a4paper]{amsart}
\usepackage{graphicx,amsfonts,amssymb,amsmath}
\usepackage[all]{xypic}
\begin{document}


\newcommand{\yiiII}{\mbox{$
\begin{picture}(13,10)(-1,1)
\put(0,10){\line(1,0){10}}
\put(0,5){\line(1,0){10}}
\put(0,0){\line(1,0){10}}
\put(0,0){\line(0,1){10}}
\put(10,0){\line(0,1){10}}
\put(5,0){\line(0,1){10}}
\end{picture}
$}}

\newcommand{\yiIIIi}{\mbox{$
\begin{picture}(13,15)(-1,1)
\put(0,15){\line(1,0){10}}
\put(0,10){\line(1,0){10}}
\put(0,5){\line(1,0){5}}
\put(0,0){\line(1,0){5}}
\put(0,0){\line(0,1){15}}
\put(10,10){\line(0,1){5}}
\put(5,0){\line(0,1){15}}
\end{picture}
$}}


\newcommand{\yii}{ \mbox{$
\mbox{$\begin{picture}(13,17)(-1,-1)
\put(0,15){\line(1,0){10}}
\put(0,10){\line(1,0){10}}
\put(0,0){\line(1,0){10}}
\put(0,-5){\line(1,0){10}}
%
\put(0,-5){\line(0,1){20}}
\put(10,-5){\line(0,1){20}}
\put(5,-5){\line(0,1){7}}
\put(5,8){\line(0,1){7}}
\put(1,2){.}
\put(1,4){.}
\put(1,6){.}
\put(6,2){.}
\put(6,4){.}
\put(6,6){.}
\end{picture}$}
$}}

\newcommand{\yiik}[1]{ \mbox{$\mbox{$\begin{picture}(11,13)(-1,-1)
\put(1,3){\mbox{\tiny ${#1}$}}
\put(4,9){\vector(0,1){6}}
 \put(4,1){\vector(0,-1){6}}
\end{picture}$}
\mbox{$\begin{picture}(13,13)(-1,-1)
\put(0,15){\line(1,0){10}}
\put(0,10){\line(1,0){10}}
\put(0,0){\line(1,0){10}}
\put(0,-5){\line(1,0){10}}
%
\put(0,-5){\line(0,1){20}}
\put(10,-5){\line(0,1){20}}
\put(5,-5){\line(0,1){7}}
\put(5,8){\line(0,1){7}}
\put(1,2){.}
\put(1,4){.}
\put(1,6){.}
\put(6,2){.}
\put(6,4){.}
\put(6,6){.}
\end{picture}$}
$}}

\newcommand{\yiikwide}[1]{ \mbox{$\mbox{$\begin{picture}(21,13)(-1,-1)
\put(1,3){\mbox{\tiny ${#1}$}}
\put(10,9){\vector(0,1){6}}
 \put(10,1){\vector(0,-1){6}}
\end{picture}$}
\mbox{$\begin{picture}(13,13)(-1,-1)
\put(0,15){\line(1,0){10}}
\put(0,10){\line(1,0){10}}
\put(0,0){\line(1,0){10}}
\put(0,-5){\line(1,0){10}}
%
\put(0,-5){\line(0,1){20}}
\put(10,-5){\line(0,1){20}}
\put(5,-5){\line(0,1){7}}
\put(5,8){\line(0,1){7}}
\put(1,2){.}
\put(1,4){.}
\put(1,6){.}
\put(6,2){.}
\put(6,4){.}
\put(6,6){.}
\end{picture}$}
$}}

\newcommand{\yiiIkwide}[1]{ \mbox{$\mbox{$\begin{picture}(21,13)(-1,-1)
\put(1,3){\mbox{\tiny ${#1}$}}
\put(10,9){\vector(0,1){6}}
 \put(10,1){\vector(0,-1){6}}
\end{picture}$}
\mbox{$\begin{picture}(19,13)(-1,-1)
\put(0,15){\line(1,0){15}}
\put(0,10){\line(1,0){15}}
\put(0,0){\line(1,0){10}}
\put(0,-5){\line(1,0){10}}
%
\put(0,-5){\line(0,1){20}}
\put(10,-5){\line(0,1){20}}
\put(15,10){\line(0,1){5}}
\put(5,-5){\line(0,1){7}}
\put(5,8){\line(0,1){7}}
\put(1,2){.}
\put(1,4){.}
\put(1,6){.}
\put(6,2){.}
\put(6,4){.}
\put(6,6){.}
\end{picture}$}
$}}

\newcommand{\yiia}{ \mbox{$
\mbox{$\begin{picture}(13,17)(-1,-5)
\put(0,15){\line(1,0){10}}
\put(0,10){\line(1,0){10}}
\put(0,0){\line(1,0){10}}
\put(0,-5){\line(1,0){10}}
\put(0,-10){\line(1,0){5}}
%
\put(0,-10){\line(0,1){25}}
\put(10,-5){\line(0,1){20}}
\put(5,-10){\line(0,1){12}}
\put(5,8){\line(0,1){7}}
\put(1,2){.}
\put(1,4){.}
\put(1,6){.}
\put(6,2){.}
\put(6,4){.}
\put(6,6){.}
\end{picture}$}
$}}

\newcommand{\yiiaa}{ \mbox{$
\mbox{$\begin{picture}(13,19)(-1,-6)
\put(0,15){\line(1,0){10}}
\put(0,10){\line(1,0){10}}
\put(0,0){\line(1,0){10}}
\put(0,-5){\line(1,0){10}}
\put(0,-10){\line(1,0){5}}
\put(0,-15){\line(1,0){5}}
%
\put(0,-15){\line(0,1){30}}
\put(10,-5){\line(0,1){20}}
\put(5,-15){\line(0,1){17}}
\put(5,8){\line(0,1){7}}
\put(1,2){.}
\put(1,4){.}
\put(1,6){.}
\put(6,2){.}
\put(6,4){.}
\put(6,6){.}
\end{picture}$}
$}}


\renewcommand{\arraystretch}{1,5}

\newcommand{\ce}{{\mathcal E}}
\newcommand{\ph}{\phantom{'}}

\newcommand{\Do}{\mbox{\sf D}}
\newcommand{\M}{\mbox{\sf M}}
\newcommand{\Pe}{{\sf P}}

\newcommand{\ep}{{\epsilon}}
\newcommand{\dt}{{\delta}}
\newcommand{\al}{{\alpha}}
\newcommand{\bt}{{\beta}}
\newcommand{\fii}{{\varphi}}
\newcommand{\om}{\omega}
\newcommand{\Om}{\Omega}
\newcommand{\sig}{\sigma}

\newcommand{\asl}{\mathfrak{sl}}
\newcommand{\agl}{\mathfrak{gl}}
\newcommand{\g}{{\mathfrak g}}
\newcommand{\G}{{\mathcal G}}

\newcommand{\C}{{\mathbb C}}
\newcommand{\R}{{\mathbb R}}

\newcommand{\Ca}{{\mathcal{C}}}
\newcommand{\id}{{\textup{id}}}
\newcommand{\Alt}{{\operatorname{Alt}}}
\newcommand{\Ad}{{\operatorname{Ad}}}

\newcommand{\bal}{\boldsymbol{\al}}
\newcommand{\bbt}{\boldsymbol{\bt}}
\newcommand{\dbal}{{\bf \dot{\bal}}}
\newcommand{\ddbal}{{\bf \ddot{\bal}}}
\newcommand{\dbbt}{{\bf \dot{\bbt}}}
\newcommand{\ddbbt}{{\bf \ddot{\bbt}}}

\newcommand{\bA}{{\bf A}}
\newcommand{\bB}{{\bf B}}
\newcommand{\dbA}{{\bf \dot{A}}}
\newcommand{\ddbA}{{\bf \ddot{A}}}
\newcommand{\dbB}{{\bf \dot{B}}}
\newcommand{\ddbB}{{\bf \ddot{B}}}
\newcommand{\dddbB}{{\bf \dddot{B}}}

\newcommand{\bE}{{\bf E}}
\newcommand{\bF}{{\bf F}}
\newcommand{\dbE}{{\bf \dot{E}}}
\newcommand{\ddbE}{{\bf \ddot{E}}}
\newcommand{\dbF}{{\bf \dot{F}}}
\newcommand{\ddbF}{{\bf \ddot{F}}}

\newcommand{\XX}{{\mathbb X}}
\newcommand{\WW}{{\mathbb W}}
\newcommand{\YY}{{\mathbb Y}}

\newcommand{\lpl}
{\mbox{$
\begin{picture}(12.7,8)(-.5,-1)
\put(2,0.1){$+$}
\put(6.2,2.5){\oval(8,8)[l]}
\end{picture}$}}

\newcommand{\tpl}
{\mbox{$
\begin{picture}(12.7,8)(-.5,-1)
\put(2.2,0.2){$+$}
\put(6,2.8){\oval(8,8)[t]}
\end{picture}$}}


\newcounter{counter}
\numberwithin{counter}{section}

 \newtheorem{lem}[counter]{Lemma}
 \newtheorem{cor}[counter]{Corollary} 
 \newtheorem{prop}[counter]{Proposition}
 \newtheorem*{prop*}{Proposition}
 \newtheorem{rem}[counter]{Remark}
 \newtheorem*{rem*}{Remark}


\title{Nonstandard operators in Grassmannian geometry}
\author{ALE\v S N\'AVRAT}
\address{Dept. of Mathematics and Statistics, Masaryk University, Kotl\' a\v rsk\' a 2, 611 37 Brno, Czech Republic, E-mail: navrat@math.muni.cz}

\maketitle

\begin{abstract}
We present a construction of curved analogues of the nonstandard operators on Grassmannians parallel to the construction of the Paneitz operator in \cite{CCCex}, but technically more demanding.
In particular, the construction breaks down in the presence of torsion.
In the second part, we prove that the nonstandard operators are not strongly invariant.
\end{abstract}

\section{Introduction}

The curved Casimir operators for parabolic geometries were originally introduced in \cite{CCCBGG}. As  described in that paper, their strong naturality properties together with the fact that they act by a multiplication by a scalar on irreducible bundles enable us to use them to construct higher order invariant operators. Concrete examples of such constructions in conformal geometry are discussed in \cite{CCCex}. 
The purpose of this article is to apply curved Casimir operators in Grassmannian  geometry to construct a family of fourth order operators which are intrinsic to any Grassmannian structure and which coincide with nonstandard operators on locally flat structures. The existence of such operators was  proved by a different method by Slov\' ak and Gover in \cite[theorem 5.1]{Gover}.  The approach via curved Casimirs gives an alternative proof of the existence and also yields new formulae for the operators, c.f. corollary \ref{factorization}. In the second part we prove by algebraic methods that these operators have the exceptional property that they are not strongly invariant.

The original inspiration comes from conformal geometry. Namely, 
the equivalence of  Grassmannian structures and  split signature conformal structures in dimension four is well known. It is obvious then that the curved version of the nonstandard operator corresponds to  the conformal square of the Laplacian $\Delta^2:\ce\to\ce[-4]$. Hence the construction of $\Delta^2$ via curved Casimirs  described in \cite{CCCex} gives us a recipe how to construct  curved analogues of nonstandard operators in Grassmannian geometry. The construction is  rather subtle  since  the direct application of the construction scheme described in \cite{CCCBGG} yields a trivial operator. This reflects the fact that $\Delta^2$ is not strongly invariant in dimension four. 

The structure of the paper is as follows. First we recall basic facts about Grassmannian geometry following from the general theory of parabolic geometries. It includes also the definitions of curved Casimir operators and nonstandard operators.
The first part of the second section concerns the initial tractor bundle for the curved Casimir procedure. We derive its composition series, an explicit form of the action of one-forms on that bundle and Casimir eigenvalues on its irreducible pieces. This is used in the second part to state and prove proposition \ref{thm} which is the first main result of this paper. The third section is  rather independent on the previous two sections. It contains the second main result,  stated in proposition \ref{weak_invariance}.


\subsection{Grassmannian structures of type $(2,n)$} 
\label{GS}
We shall use the conventions and the abstract index notation of \cite{prolongation} and \cite{Gover}. An almost Grassmannian structure of type $(2,n)$ on a smooth manifold $M$ of dimension $2n$ is given by two auxiliary vector bundles $\ce^A$, $\ce^{A'}$ of ranks $n$ and $2$ respectively, and identifications
\begin{equation}
\label{AG}
TM=\ce^a\cong\ce_{A'}\otimes\ce^A=\ce_{A'}^A, \quad
\Lambda^2\ce_{A'}\cong\Lambda^n\ce^A.
\end{equation}
Equivalently, the structure is a classical first order G-structure with reduction of the structure group $GL(2n,\R)$ to its subgroup $S(GL(2,\R)\times GL(n,\R))$. 
It is well known that one can construct a unique Cartan connection associated to this structure which makes it into the $|1|$-graded normal parabolic geometry of type $(G,P)$, where $G=SL(2+n,\R)$ and $P$ is the stabilizer of $\R^2$ in $\R^{2+n}$. The corresponding flat model $G/P$ are Grassmannians $\operatorname{Gr}_2(\R^{2+n})$. In the diagram notation, the $|1|$-grading of Lie algebra $\g=\asl(2+n,\R)$ defining the geometry is given by the $A_{n+1}$-diagram with  the second node crossed. Hence $\g=\g_{-1}\oplus\g_0\oplus\g_1$, where
$$
\g_0\cong\mathfrak{s}(\agl (2,\R)\oplus\agl (n,\R)),\quad \g_1\cong(\g_{-1})^* \cong\R^2\otimes\R^{n*}.
$$
In general, the full obstruction against local flatness of a normal parabolic geometry is encoded in the harmonic curvature $\kappa_H$. In the case of the almost Grassmannian geometry of type $(2,n)$, this curvature consists of two parts: the homogeneity one component $(\kappa_H)_1$ and the homogeneity two component $(\kappa_H)_2$. For detailed description  see \cite[Section 4.1.3]{parabook}. The former one can be interpreted geometrically as the torsion of a linear connection on the tangent bundle. The connections with such torsion $T=(\kappa_H)_1$ form a class of distinguished connections modelled on one-forms. They are also called Weyl connections and can be viewed as analogues of Levi-Civita connections in conformal geometry, see \cite[section 5.1]{parabook}. The  harmonic curvature component $(\kappa_H)_2$ is  a component of the curvature of any of these  connections.  

A Grassmannian geometry is usually defined as an almost Grassmannian geometry which admits a torsion free connection moreover. 
By \cite[Theorem 4.1.1]{parabook}, the full obstruction against existence of such a connection is exactly the harmonic curvature component $(\kappa_H)_1$. Thus a Grassmannian geometry can be equivalently described as an almost Grassmannian geometry with $(\kappa_H)_1=0$ and hence it is also called semi flat in the literature. Obviously, the distinguished connections in such a case are exactly the torsion-free connections compatible with the isomorphisms in \eqref{AG}. In subsequent formulas for differential operators, we will use such distinguished connections which  moreover induce a flat connection on $\Lambda^2\ce_{A'}\cong\Lambda^n\ce^A$. They are called closed since they form a subclass  modeled on closed one forms. For more details see \cite[section 5.1.7]{parabook}.

Let us note that there is another geometry with a similar behavior, namely the  quaternionic geometry. It is given by the same grading as the Grassmannian geometry but on the quaternionic real form of $\g^\C$ rather than the split form. Nevertheless, considering complexifications and complex bundles, the identifications \eqref{AG} are satisfied. Thus we may include this geometry into our framework and all the constructions and results hold also for quaternionic geometry.

\subsection{Curvature of distinguished connections}
Theorem 4.1.1 in \cite{parabook} also states that for any normal $|1|$-graded parabolic geometry the equation $(\kappa_H)_1=0$ implies that the harmonic curvature component $(\kappa_H)_2$ coincides with the so-called Weyl curvature $W:=R+\partial(\Pe)$. Here, $R$ is the usual curvature tensor while $\partial$ is the bundle map induced by the Lie algebra differential and $\Pe\in\Om^1(M,T^*M)$ is a uniquely defined piece of curvature called Rho-tensor. 
Hence according to the description of the harmonic curvature in Grassmannian geometry, the Weyl curvature is given by $W_{ab}{}^d{}_{c}=W^{A'B'D}_{A\ph B\ph C}\dt^{C'}_{D'}$, where 
\begin{equation}
\label{W_sym}
W^{A'B'D}_{A\ph B\ph C}\in \Gamma(\ce^{[A'B']}\otimes(\ce_{(AB C)}^{\ph\ph\ph\ph\ph D})_0),
\end{equation}
and where $()_0$ denotes the trace free part. See e.g. \cite[section 4.1.3]{parabook}. Thus the whole curvature $R_{ab}{}^d{}_{c}$ of a distinguished connection $\nabla_a$ is given by the sum of such $W_{ab}{}^d{}_{c}$ and terms which are linear and zero order in $\Pe_{ab}$. The former part is obviously invariant while the latter depend on the  choice of $\nabla_a$. Namely, if the change from $\nabla_a$ to another distinguished connection  $\hat{\nabla}_a$ is described by one form $\Upsilon_a$,  then the linearized change of the Rho-tensor is given by
$\hat{\Pe}_{ab}=\Pe_{ab}+\nabla_a\Upsilon_b.$
See e.g. \cite{Gover} or section 5.1.8 of \cite{parabook}.
Moreover, it follows from the algebraic Bianchi identity that the Rho-tensor corresponding to a closed distinguished connection is symmetric, see equation (9) in \cite{Gover}.
The next consequence of theorem 4.1.1 in \cite{parabook} is that the tensor $d^\nabla\Pe$, called the Cotton-York tensor, can be expressed in terms of $W$. The exact formula can be deduced from the differential Bianchi identity, see \cite[Lemma A.5]{thesis}. In particular, it follows that  
\begin{equation}
\label{Pe_sym}
(d^\nabla\Pe)^{A'B'C'}_{A\ph B\ph C}\in \Gamma(\ce^{[A'B']C'}\otimes\ce_{(AB C)}).
\end{equation}

\subsection{Grassmannian bundles}
\label{Grassmannian bundles}
Following \cite{Gover} and \cite{prolongation}, we adopt the convention $\ce[-1]:=\Lambda^2\ce^{A'}\cong\Lambda^n\ce_A,$ and then for $w\in\mathbb Z_-$ we put $\ce[w]:=\ce[-1]^{-w}$ and $\ce[-w]:=\ce[w]^*$. In analogy with conformal geometry, the bundle $\ce[w]$ will be called the bundle of densities of weight $w$, and adding $[w]$ to the notation for a bundle indicates a tensor product by $\ce[w]$. Motivated by the case $n=2$ when the structure is equivalent to the spin conformal structure, the basic bundles $\ce^A$ and $\ce_{A'}$ are called spinor bundles, and the corresponding indices are called spinor indices. They can be raised and lowered similarly as the abstract indices in conformal or Riemannian geometry. However, we use a skew symmetric object instead of a metric which is  symmetric, and thus the order of indices is important. Namely, to raise primed indices, we will use the canonical section  $\ep^{A'B'}$ of $\ce^{[A'B']}[1]$ which gives the isomorphism   $\ce[-1]\ni f\mapsto f\ep^{A'B'}\in\ce^{[A'B']}$, and we use it in such a way that $v^{B'}=v_{A'}\ep^{A'B'}.$ The inverse $\ep_{A'B'}$ is used to lower primed indices as follows $v_{B'}=v^{A'}\ep_{A'B'}.$  

The Grassmannian standard tractor bundle will be denoted by $\ce^\al$. It is a vector bundle of rank $n+2$ induced by the standard representation of $\asl(2+n)$ on $\R^{2+n}$. The obvious filtration $\R^{2+n}\supset\R^2$ gives rise to the filtration $\ce^\al\supset\ce^{A'}$, where $\ce^\al/\ce^{A'}\cong\ce^A$. Similarly, the standard cotractor bundle $\ce_\al$ is endowed with filtration $\ce_\al\supset\ce_A$, where  $\ce_\al/\ce_A\cong\ce_{A'}$. These data
define a composition series for $\ce_{\al}$ that will be denoted by $\ce_{\al}=\ce_{A'}\lpl\ce_A$ henceforth. This notation is motivated by the fact that summands include while there is a projection onto direct summands. Under a choice of a torsion-free connection $\nabla_a$, the composition series splits.
The exact formulas for transformations of the splittings of tractor bundles under a change from $\nabla_a$ to another torsion-free connection $\hat{\nabla}_a$ can be found by making explicit the general formulas in \cite[Section 5.1]{parabook}, see e.g. \cite{Gover}.
However, we will not need these formulas in this article since we will always deal with objects and operations which are known to be invariant. The splittings will be used only to compute explicit formulas for these objects and operations.

Following \cite{CCCex} and \cite{prolongation}, sections of tractor bundles will be written either as a formula in injectors $X^A_{\al}\in \ce^A_\al$, $Y^{A'}_{\al}\in \ce^{A'}_\al$ or simply denoted by a "vector". For example, a section $v_\al$ of the standard cotractor bundle reads as
\begin{equation}
\label{standard_cotractor}
v_{\al}=v_{A'}Y^{A'}_{\al}+v_A X^A_{\al}
=\begin{pmatrix}v_{A'}\\v_A\end{pmatrix},
\end{equation}
where $v_{A'}$ and $X^A_\al$ are invariant while $v_A$ and $Y^{A'}_\al$ depend on a choice of splitting.

\subsection{Notation for forms and irreducible bundles}
\label{Notation}
To simplify subsequent  expressions let us adopt the index notation for forms of \cite{prolongation}. The usual notation $\ce_{[AB\dots C]}$ for $k$th skew symmetric power of $\ce_A$ will be abbreviated by the following multi-indices 
$$
\begin{array}{l}
\bA^k:=[A^1\cdots A^k], \quad k\geq 0,
\\
\dbA^k:=[A^2\cdots A^k], \quad k\geq 1,
\\
\ddbA^k:=[A^3\cdots A^k], \quad k\geq 2.
\end{array}
$$
The subscript $k$ will be usually omitted, i.e.  $\ce_\bA$ means $\ce_{\bA^k}$ throughout the article. Also the square bracket may be absent. Indices labelled with sequential superscripts automatically indicate a completely skew set of indices. The same notation will be used also for skew symmetric powers of tractor bundles. For example, the bundle $\ce_{\al^2\cdots \al^k}=\ce_{[\al^2\cdots \al^k]}$ will be denoted by $\ce_{\dbal}$.
We will combine this notation  with the usual Young diagram notation for irreducible bundles induced by $SL(n)$-modules as follows. 
We adopt the convention that we symmetrized over sets of indices corresponding to the rows of the diagram first and then with the result skew over sets of indices corresponding to the columns of the diagram. The sets of indices in columns may be then shortened by multi-indices as above. 
For instance suppose we have a general valence four spinor 
$A_{A^1A^2B^1B^2}\in\ce_{A^1A^2B^1B^2}[w].$ 
If we first symmetrized as follows 
$B_{A^1A^2B^1B^2}:=A_{(A^1|A^2|B^1)B^2},$
$C_{A^1A^2B^1B^2}:=B_{A^1(A^2|B^1|B^2)}$
and then on this result skew over the nonsymmetric indices
$D_{\bA^2\bB^2}:=C_{[A^1A^2][B^1B^2]},$
then $$D_{\bA^2\bB^2}\in\yiiII\ce_{\bA^2\bB^2}[w].$$

\subsection{Nonstandard operators}
\label{NO}
Using the above notation, we can explicitly write down irreducible
subbundles of $k$-forms. Namely, since $T^*M=\ce^{A'}\otimes\ce_A$ by \eqref{AG}, each such subbundle is given by a product of an irreducible $SL(2)$-bundle in $\otimes^k\ce^{A'}$, which is a symmetric power of $\ce^{A'}$, and an irreducible $SL(n)$-bundle in $\otimes^k\ce_{A}$ indicated by the corresponding diagram. 
For example, the bundle of two-forms splits into  $\ce^{(A'B')}\otimes\ce_{\bA^2}$ and $\ce_{(AB)}[-1]$, while the decomposition of $\Lambda^4T^*M$ reads as  
$$
\Lambda^4(\ce^{A'}_A)=\yiiII\ce_{\bA^2\bB^2}[-2]\oplus\ce^{(A'B')}\otimes\yiIIIi\ce_{\bA^3\bB^1}[-1]\oplus\ce^{(A'B'C'D')}_{\bA^4}.
$$
The decomposition of $\Lambda^kT^*M$  can be written a similar form for any $k$. In particular, the number of columns of Young diagrams that occur is never more  than two since this would correspond to a $k$-form which would be skew symmetric in more than two primed spinor indices, and such forms vanish due to the low rank of $\ce^{A'}$. For instance, the irreducible bundles of $2k$-forms are bundles which are given for each $0\leq \ell\leq k$ by the tensor product of $S^{2\ell}\ce^{A'}[\ell-k]$ with the subbundle of $\otimes^{2k}\ce_A$ which is given by the Young diagram with columns of height $k+\ell$ and $k-\ell$. Of course, if $k+\ell>n$ then the component vanishes, and if $k+\ell=n$ then we get a one more copy of $\ce[-1]\cong\Lambda^n\ce_A$ and the Young diagram contains $n$ boxes less. In particular, we see that for $2k\leq n$, the bundle of $2k$-forms decomposes into $k+1$ components, while for $2k>n$ we have $n-k+1$ components. The decomposition of odd degree forms is similar, and so the De Rham sequence splits into a triangular pattern, see \cite{Subcomplexes} for more details.

 We mainly are interested in the long side of this triangle. In the case of a locally flat Grassmannian structure, there occur invariant operators which correspond to nonstandard homomorphisms of generalized Verma modules, cf. \cite{Lepowsky} and \cite{Boe}.  Concretely, for each $2\leq k\leq n$ there is a fourth order invariant operator 
$$
\widetilde{\square}_{ABCD}: \yiikwide{k-2}\ce_{\ddbA\ddbB}[-k+2]\to\yiik{k}\ce_{\bA\bB}[-k].
$$
These operators are called the nonstandard operators for Grassmannian structures. In terms of preferred flat connections (which are known to exist on locally flat structures), the operator $\widetilde{\square}_{ABCD}$ is obviously given by
$
\nabla_{I'A}\nabla^{I'}_B\nabla_{J'C}\nabla^{J'}_D,
$
followed by the projection to the target bundle. Let us note that such projection is unique since the commutativity of $\nabla^{A'}_A$ ensures that
$
\nabla_{I'A}\nabla^{I'}_B\nabla_{J'C}\nabla^{J'}_D$
is a section of 
$\ce_{[AB][CD]}[-2],
$
and the target bundle appears only in the tensor product of the initial bundle with $\yiiII\ce_{ABCD}[-2]$ which occurs with multiplicity one in $\ce_{[AB][CD]}[-2]$.

All other invariant operators between exterior forms are components of exterior derivative and their nonzero compositions (of order two). Therefore, they exist also on general almost Grassmannian structures. On the other hand, it may be proved that the nonstandard operators are not strongly invariant, see \cite{Navrat}. Hence it is not clear whether these extend to invariant operators to a larger class of (almost) Grassmannian structures. However, it was proved by J. Slov\' ak and R. Gover in \cite[Theorem 5.1]{Gover} that they all do exist on any Grassmannian structure of type $(2,n)$. In sequel, we  construct them via so called curved Casimir operators, and thus we give an alternative proof of their existence and alternative formulae for them.

\subsection{A formula for the curved Casimir operator, construction principle}
First recall from Theorem 3.4 of \cite{CCCBGG} that the curved Casimir $\Ca$ is and invariant operator which on an irreducible bundle $W\to M$ acts by a real multiple of the identity. We denote the corresponding scalar by $\beta_W$ and we call it the Casimir eigenvalue. By the theorem, this scalar can be computed in terms of weights of the representation which induces $W$. Namely, if the lowest weight of this representation is $-\lambda$, then
\begin{equation}
\label{eigenvalue}
\beta_W=\langle\lambda,\lambda+2\rho\rangle,
\end{equation}
where $\rho$ is the lowest form which by definition equals half the sum of all positive roots or equivalently, it equals the sum of all fundamental weights. 

A suitable formula for $\Ca$ on an arbitrary natural Grassmannian bundle $V\to M$ can be obtained from the formula in terms of an adapted local frame for the adjoint tractor bundle from Proposition 3.3 of \cite{CCCBGG}. Following the proof of  Proposition 2.2 in \cite{CCCex}, which gives a formula for $\Ca$ in conformal geometry, one gets a formula which has precisely the same form as the formula in this proposition. Namely, having fixed a connection $\nabla$ from the class of the distinguished connections, the adjoint tractor bundle splits as $TM\oplus\operatorname{End}_0(TM)\oplus T^*M$, and having chosen a local orthonormal frame $\xi_\ell$ for $TM$ with dual frame $\varphi^\ell$ for $T^*M$ (with respect to the Cartan-Killing form), the formula for the curved Casimir operator  reads as
\begin{equation}
\label{Ca_short}
\Ca(s)=\beta(s)-2\sum_\ell\fii^\ell\bullet \nabla_{\psi_\ell}s-2\sum_\ell\fii^\ell\bullet\Pe (\psi_\ell)\bullet s,
\end{equation}
where $\beta:V\to V$ is the bundle map which acts on each irreducible component $W\subset V$ by multiplication by $\beta_W$.
For the sake of simplicity, we  write this formula in a shortened form as $\Ca =\beta_W-2\nabla\bullet -2\Pe\bullet\bullet\ph.$

The construction principle we use is inspired by the construction of splitting operators in section 3.5 of \cite{CCCBGG}. 
Let $\mathcal{T}\to M$ be a tractor bundle. The filtration of the standard tractor bundle induces a natural filtration $\mathcal{T}=\mathcal{T}^0\supset\mathcal{T}^1\supset
\cdots\supset\mathcal{T}^N$, which we write as $\mathcal{T}=\mathcal{T}^0/\mathcal{T}^1\lpl\mathcal{T}^1/\mathcal{T}^2
\lpl\cdots\lpl\mathcal{T}^N.$
Each of the subquotients $\mathcal{T}^i/\mathcal{T}^{i+1}$ splits into a direct sum of irreducible tensor bundles. We know from above that on sections of each such irreducible $W\subset\mathcal{T}^i/\mathcal{T}^{i+1}$, the curved Casimir acts by multiplication with $\beta_W$. Let $\beta_j^1,\dots,\beta_j^{n_j}$ denote for all $j>i$ the different eigenvalues that occur in the decomposition of $\mathcal{T}^j/\mathcal{T}^{j+1}$.
Then the natural operator on $\Gamma(VM)$ defined by $L:=\Pi_{j=i+1}^N\Pi_{\ell=1}^{n_j}(\Ca-\beta^\ell_j)$ descends to an operator $\Gamma(WM)\to\Gamma(V^iM)$. Moreover, if for $j>i$ all eigenvalues $\beta^\ell_j$ are different from $\beta_W$, then $L$ defines a natural splitting operator. On the other hand, if an eigenvalue $\beta^\ell_j$ corresponding to an irreducible bundle $W'M\subset V^jM/V^{j+1}M$  coincides with $\beta_W$, then $L$ restricted to $WM$ defines an invariant operator  $\Gamma(WM)\to\Gamma(W'M)$. A detailed explanation of the construction principle  can be found in section 2.3 of \cite{CCCex}.

\section{Curved Casimir construction }

The first step of the construction of nonstandard operators is a choice of a suitable tractor bundle. Following the construction of Paneitz operator in \cite{CCCex}, it is natural to take as the input bundle such a tractor bundle, which has the initial bundle and the target bundle of the operator in the top slot and the bottom slot respectively. Hence from the description of $\widetilde{\square}_{ABCD}$ in section \ref{NO} we conclude that the right input tractor bundle for our construction is 
\begin{equation}
\label{tractor_bundle}
\mathcal{T}:=\yiik{k}\ce_{\bal\bbt}[-k].
\end{equation}

\subsection{The composition series of $\mathcal{T}$}
The composition series for this tractor bundle can be computed from the series of the simpler bundles. Namely, it follows from the structure of the standard cotractor bundle \eqref{standard_cotractor} that $\ce_{[\al\bt]}=\ce_{[A'B']}\lpl \ce_{A'B}\lpl\ce_{[AB]}$. In terms of densities and multi indices, defined in \ref{Grassmannian bundles} and   \ref{Notation} respectively, this can be also written as $\ce_{\bal^2}=\ce[1]\lpl \ce^{A'}_{A}[1]\lpl\ce_{\bA^2}.$ It is then easy to see that  for an arbitrary $k$ the composition series of $k$-forms is 
$$
\ce_{\bal}=\ce_{\ddot{\bA}}[1]\lpl \ce^{A'}_{\dot{\bA}}[1]\lpl\ce_\bA.
$$
From this composition series, one easily deduces that the second symmetric power satisfies
$$
\ce_{(\bal\bbt)}=
\ce_{(\ddbA\ddbB)}[2]
\lpl
\ce^{A'}_{\dbA\ddbB}[2]
\lpl
(
\ce^{(A'B')}_{(\dbA\dbB)}[2]
\oplus
\ce_{[\dbA\dbB]}[1]
\oplus
\ce_{\bA\ddbB}[1]
)
\lpl
\ce^{A'}_{\bA\dbB}[1]
\lpl
\ce_{(\bA\bB)}.
$$
The  bundles displayed are not irreducible however. On the left-hand side, there appears $\mathcal{T}$ as a direct summand. It is an easy observation that the diagrams corresponding to the other summands have also two columns at most which means that the height of one of them is greater then $k$. Hence the desired composition series of $\mathcal{T}$ can be detected in the right-hand side of the previous equation by considering only such terms that vanish under alternations over more than $k$ tractor indices. Its exact form is given in the lemma below.

In order to write down the representatives of the the bundles in the composition series explicitly, it is convenient to use a notation analogous to the notation of \cite{prolongation}. Concretely, we use the injectors $X^A_\al$, $Y^{A'}_\al$ from \eqref{standard_cotractor} to define new injectors
\begin{equation}
\label{injectors}
\begin{array}{l}
\XX^{\bA}_{\bal}:=X^{A_1}_{[\al_1}\dots X^{A_k}_{\al_k]}\in\ce^{\bA}_{\bal},
\\
\WW^{A'\dbA}_{\bal}:=Y^{A'}_{[\al_1}X^{A_2}_{\al_2}\dots X^{A_k}_{\al_k]}\in\ce^{A'\dbA}_{\bal},
\\
\YY^{A'B'\ddbA}_{\bal}:=Y^{A'}_{[\al_1}Y^{B'}_{\al_2}X^{A_3}_{\al_3}\dots X^{A_k}_{\al_k]}\in\ce^{[A'B']\ddbA}_{\bal}.
\end{array}
\end{equation}
In terms of these injectors, a tractor $k$-form $\fii_{\bal}$ can be simply expressed  as 
$$
\fii_{\bal}=
\sigma_{\ddbA}\ep_{A'B'}\YY^{A'B'\ddbA}_{\bal}
+
\mu^{A'}_{\dbA}\ep_{A'B'}\WW^{B'\dbA}_{\bal}
+
\rho_{\bA}\XX^{\bA}_{\bal},
$$
where $\sigma_{\ddbA}\in\ce_{\ddot{\bA}}[1]$, $\mu^{A'}_{\dbA}\in\ce^{A'}_{\dot{\bA}}[1]$, $\rho_{\bA}\in\ce_\bA$. In a similar way, one can explicitly write down the sections of $\ce_{(\bal\bbt)}$ and also the sections of its subbundles, in particular the sections of our bundle $\mathcal{T}$.

\begin{figure}[htb]
\label{composition_series}
\caption{Composition series of $\mathcal{T}$}
\begin{equation}
\label{comp_series}
\mathcal{T}:=\yii\ce_{\bal\bbt}[-k]=
\begin{pmatrix}
 \yii\ce_{\ddbA\ddbB}[-k+2]
\\
 \tpl
\\
\ce^{A'}\otimes\yiia\ce_{\dbA\ddbB}[-k+2]
\\
\tpl
\\
\ce^{(A'B')}\otimes\yii\ce_{\dbA\dbB}[-k+2]
\oplus
\yiiaa\ce_{\bA\ddbB}[-k+1]
\\                                                             
\tpl
\\
\ce^{A'}\otimes\yiia\ce_{\bA\dbB}[-k+1]
\\
\tpl
\\
\yii\ce_{\bA\bB}[-k]
\end{pmatrix}
\end{equation}
\end{figure}

\begin{lem}
The composition series of tractor bundle $\mathcal{T}$ from \eqref{tractor_bundle} has the form as displayed in figure \ref{composition_series}. In terms of injectors defined by \eqref{injectors}, its section $v_{\bal\bbt}$ can be expressed as 
\begin{equation}
\label{comp_injectors}
\begin{array}{l}
v_{\bal\bbt}=
\sig_{\ddbA\ddbB}\ep_{A'B'}\ep_{C'D'}\YY^{A'B'\ddbA }_{(\bal}\YY^{C'D'\ddbB }_{\bbt)}
+
\mu^{A'}_{\dbA\ddbB}\ep_{A'B'}\ep_{C'D'}\WW^{B'\dbA }_{(\bal}\YY^{C'D'\ddbB }_{\bbt)}
\\
+
A^{A'B'}_{\dbA\dbB}\ep_{A'C'}\ep_{B'D'}\WW^{C'\dbA }_{(\bal}\WW^{D'\dbB }_{\bbt)}
+
\al_{\bA\ddbB}\ep_{C'D'}(\XX^{\bA}_{(\bal}\YY^{C'D'\ddbB }_{\bbt)}
+\frac{k}{2}\WW^{C'A_1\ddbB}_{(\bal}\WW^{D'\dbA}_{\bbt)})
\\
+
\nu^{A'}_{\bA\dbB}\ep_{A'B'}\XX^{\bA}_{(\bal}\WW^{B' \dbB }_{\bbt)}
+
\rho_{\bA\bB}\XX^{\bA}_{(\bal}\XX^{\bB}_{\bbt)},
\end{array}
\end{equation}
where $\sig_{\ddbA\ddbB}$, $\mu^{A'}_{\dbA\ddbB}$, $A^{A'B'}_{\dbA\dbB}$, $\al_{\bA\ddbB}$, $\nu^{A'}_{\bA\dbB}$, $\rho_{\bA\bB}$ are representatives of the corresponding bundles appearing in the composition series of $\mathcal{T}$.
\end{lem}
\begin{proof}
By definition, sections of our bundle are such sections of $\ce_{(\bal\bbt)}[-k]$ that vanish under all alternations over more than $k$ indices.  It is an easy observation that the low dimension of $\ce^{A'}$ implies that this condition translates from tractor indices to unprimed spinor indices. Hence the height of Young diagrams of all bundles in the composition series of our bundle is less or equal to $k$. Looking at the composition series of $\ce_{(\bal\bbt)}$, this condition determines a unique irreducible bundle in each summand. Twisting everything by weight $-k$,  these bundles already are the bundles appearing in \eqref{comp_series}. Their representatives in our tractor bundle are obvious except for the bundle with Young diagram with columns of heights $k$ and $k-2$. One copy of this bundle sits in  $\ce_{\dbA\dbB}[-k+1]$ and the second in $\ce_{\bA\ddbB}[-k+1]$, and the condition on tractor indices yields a nontrivial relation between these two copies. Namely, we will show that the corresponding sections $B_{\dbA\dbB}$ and $\al_{\bA\ddbB}$ are related by
\begin{equation}
\label{B_al}
B_{\dbA\dbB}=\frac{k}{2}\al_{[A_2|\dbB | \ddbA]},
\quad
\al_{\bA\ddbB}=(-1)^k(k-1)B_{\bA\ddbB}.
\end{equation}
This will finish the proof since the first equation shows that a section $\al_{\bA\ddbB}$ of the bundle with Young diagram with columns of heights $k$ and $k-2$ is represented in our tractor bundle by
$$
\al_{\bA\ddbB}\ep_{C'D'}(\XX^{\bA}_{(\bal}\YY^{C'D'\ddbB }_{\bbt)}
+\frac{k}{2}\WW^{C'A_1\ddbB}_{(\bal}\WW^{D'\dbA}_{\bbt)}).
$$
Equations \eqref{B_al} relating these two sections can be found by looking at the middle slots of $v_{[\bal \bt_1]\dbbt}=0$.
Namely, expanding the alternation in $\bbt$ with respect to $\bt_1$ and $\dbbt$ yields
$$
\WW^{C'\dbA }_{[\bal}\WW^{D' \dbB }_{\bt_1]\dbbt}=
(-1)^{k-1}\frac1{k}\YY^{C'D'\dbA }_{[\bal\bt_1]}\XX^{\dbB}_{\dbbt}
+\frac1{k}\sum\limits_{i=2}^{k}(-1)^{i-1}\WW^{C'\dbA B_i }_{[\bal\bt_1]}\WW^{D'B_2\cdots\hat{B}_i\cdots B_k}_{\dbbt}
$$
and since $B_{\dbA\dbB}$ is skew symmetric in $B_2,\dots, B_k$,
the alternation over tractor indices $\al_1,\dots,\al_k$ and $\bt_1$ of term $B_{\dbA\dbB}\ep_{C'D'}\WW^{C'\dbA }_{(\bal}\WW^{D' \dbB }_{\bbt)}$ is equal to
\begin{equation}
\label{B_skew}
B_{\dbA\dbB}\ep_{C'D'}\left((-1)^{k-1}\frac1{k}\YY^{C'D'\dbA }_{[\bal\bt_1]}\XX^{\dbB}_{\dbbt}-\frac{k-1}{k}\WW^{C'\dbA B_2 }_{[\bal\bt_1]}\WW^{D'\ddbB}_{\dbbt}\right),
\end{equation}
while the alternation of term $A^{A'B'}_{\dbA\dbB}\ep_{A'C'}\ep_{B'D'}\WW^{C'\dbA }_{(\bal}\WW^{D' \dbB }_{\bbt)}$ over the same tractor indices obviously vanishes.
Similarly, for the other injectors the expansion of $\bbt$ with respect to $\bt_1$ and $\dbbt$ yields
$$
\XX^{\bA}_{[\bal}\YY^{C'D' \ddbB }_{\bt_1]\dbbt}=
(-1)^k\frac2{k}\WW^{[C'|\bA|}_{[\bal\bt_1]}\WW^{D']\ddbB}_{\dbbt}
-\frac1{k}\sum\limits_{i=3}^{k}(-1)^i\XX^{\bA B_i}_{[\bal\bt_1]}\YY^{C'D'B_3\cdots\hat{B}_i\cdots B_k}_{\dbbt}
$$
and
$$
\XX^{\bA}_{[\bt_1|\bbt|}\YY^{C'D' \ddbB }_{\bal]}=
\frac1{k}\sum\limits_{i=1}^{k}(-1)^{i-1}\YY^{C'D' \ddbB A_i}_{[\bal\bt_1]}\XX^{A_1\cdots\hat{A}_i\cdots A_k}_{\dbbt}.
$$
Concerning the first formula, the second term on the right vanishes when applied to $\al_{\bA\ddbB}$ since each summand is skew symmetric in $k+1$ unprimed indices. Skewing the second formula in $A_1,\dots,A_k$ one obtains that the alternation over tractor indices $\al_1,\dots,\al_k$ and $\bt_1$ of term
$\al_{\bA\ddbB}\ep_{C'D'}\XX^{\bA}_{(\bal}\YY^{C'D' \ddbB }_{\bbt)}$ is
\begin{equation}
\label{alpha_skew}
\al_{\bA\ddbB}\ep_{C'D'}\left((-1)^k\frac1{k}(\WW^{C'\bA}_{[\bal\bt_1]}\WW^{D'\ddbB}_{\dbbt}+\frac12\YY^{C'D' \ddbB A_1}_{[\bal\bt_1]}\XX^{\dbA}_{\dbbt}\right).
\end{equation}
Then $v_{[\bal \bt_1]\dbbt}=0$ implies that the sum of \eqref{B_skew} and \eqref{alpha_skew} is equal to zero, and this immediately gives equations \eqref{B_al}.
 \end{proof}

\subsection{The action of one forms on $\mathcal{T}$}
The dual to standard representation of $\asl(2+n)$ on $\R^{2+n}$ restricted to $\g_1$ translates to bundles and gives rise to an action of $\ce_a$ on $\ce_{\al}$.
It is easy to see that for an one-form  $\fii^{A'}_A\in\ce^{A'}_A=\ce_a$ and a standard cotractor $v_\al=v_{A'}Y^{A'}_\al+v_aX^A_\al$ its explicit form reads as $\fii\bullet v_\al=-\fii^{A'}_A v_{A'} X^A_{\alpha}$. Hence the action of $\fii^{A'}_A$ on the injectors is given by equations
$$
\begin{array}{cccc}
 (\fii\bullet Y)^{A'}_{\alpha}=-\fii^{A'}_A X^A_{\alpha},&\fii\bullet X=0.
\end{array}
$$
Using these basic relations one is able to compute easily the action of one forms on more complicated tractor bundles.
\begin{lem}
The action of an one form $\fii^{A'}_A\in\ce^{A'}_A$ on a section of  $\mathcal{T}$ is given by
\begin{equation}
\label{action}
\fii^{A'}_A\bullet
\begin{pmatrix}
\sig_{\ddbA\ddbB}
\\
\mu^{A'}_{\dbA\ddbB}
\\
A^{A'B'}_{\dbA\dbB} | \al_{\bA\ddbB}
\\
\nu^{A'}_{\bA\dbB} 
\\ 
\rho_{\bA\bB}
\end{pmatrix}
=
\begin{pmatrix}
0\\
\fii^{A'}_{[A_2}\sig_{\ddbA]\ddbB}
\\
\fii^{(A'}_{[A_2}\mu^{B')}_{|\dbB|\ddbA]}
+\fii^{(A'}_{[B_2}\mu^{B')}_{|\dbA|\ddbB]}
\quad|\quad
\fii_{I'[A_1}\mu^{I'}_{\dbA]\ddbB}
\\
2\fii_{I'[A_1} A^{I'A'}_{\dbA]\dbB} 
+2\fii^{A'}_{[B_2}\alpha_{|\bA|\ddbB]}
-k(-1)^k\fii^{A'}_{[A_1}\alpha_{\dbA]\dbB}
\\
\fii_{I'[A_1}\nu^{I'}_{|\bB|\dbA]}
+\fii_{I'[B_1}\nu^{I'}_{|\bA|\dbB]}
\end{pmatrix}
\end{equation}
\end{lem}
\begin{proof}
From the defining equations \eqref{injectors} for the injectors  $\XX^{\bA}_{\bal}$, $\WW^{A'\dbA}_{\bal}$, 
$\YY^{A'B'\ddbA}_{\bal}$ and from the basic relations for the action of an one form  $\fii^{A'}_A\in\ce^{A'}_A$ on injectors $X^A_\al,Y^{A'}_\al$, one immediately derives 
$$
\begin{array}{l}
(\fii\bullet \YY)^{A'B'\ddbA}_{\bal}=2\fii^{[A'}_{I}\WW^{B']I\ddbA}
\\
(\fii\bullet \WW)^{A'\dbA}_{\bal}=-\fii^{A'}_I\XX^{I\dbA}_{\bal}
\\
\phantom{(}
\fii\bullet\XX=0.
\end{array}
$$
Now a direct use of these equations to \eqref{comp_injectors} yields the result.
\end{proof}

\subsection{Casimir eigenvalues on $\mathcal{T}$}
Knowing the irreducible bundles occurring in the composition series of $\mathcal{T}$, we can directly compute the corresponding  Casimir eigenvalues using formula \eqref{eigenvalue}. Since the minus lowest weights of basic bundles $\ce^{A'}$, $\ce[1]$ and $\Lambda^k\ce_{A}=\ce_\bA$ are given by $\om_1-\om_2$, $\om_2$ and $\om_{k+2}-\om_2$ respectively, where $\om_k$ for $k=1,\dots, n-1$ denotes fundamental weights of $\g$,  the minus lowest weights  corresponding to the bundles in \eqref{comp_series} are
$$
\begin{pmatrix}
&\lambda_0&
\\
&\lambda_1&
\\
\lambda_2^1&|&\lambda_2^2
\\
&\lambda_3&
\\
&\lambda_4&
\end{pmatrix}
=
\begin{pmatrix}
2\om_{k}-k\om_2
\\
\om_{k+1}+\om_{k}-(k+1)\om_2+\om_1
\\
2\om_{k+1}-(k+2)\om_2+2\om_1
\quad|\quad
\om_{k+2}+\om_k-(k+1)\om_2
\\
\om_{k+2}+\om_{k+1}-(k+2)\om_2+\om_1
\\
2\om_{k+2}-(k+2)\om_2
\end{pmatrix}.
$$
The computation of Casimir eigenvalues from formula \eqref{eigenvalue} is parallel to \cite{CCCex} using these weights as an input. We get the following.
\begin{lem}
\label{eigenvalues}
The eigenvalues of the curved Casimir operator on the irreducible bundles appearing in the composition series of $\mathcal{T}$ read as 
$$
\begin{pmatrix}
&\beta_0&
\\
&\beta_1&
\\
\beta_2^1&|&\beta_2^2
\\
&\beta_3&
\\
&\beta_4&
\end{pmatrix}
=
\begin{pmatrix}
&0&
\\
&0&
\\
4&|&-4
\\
&0&
\\
&0&
\end{pmatrix}.
$$
\end{lem}

\subsection{Construction of nonstandard operators}

Looking at Casimir eigenvalues for irreducible components of the tractor bundle $\mathcal{T}$ displayed in lemma \ref{eigenvalues} we see that the candidates for curved analogues of the nonstandard operators $\widetilde{\square}_{ABCD}$ are induced by $\Ca^3\circ(\Ca-4)\circ(\Ca+4)$. By the construction principle, such a composition of the curved Casimirs gives an invariant operator between the top slot and the bottom slot which are exactly the bundles where the nonstandard operators is defined. However, we will show that such direct construction yields always a trivial operator. 
The reason for this might be seen in the high degeneracy of this case -- there appear four zeros among the Casimir eigenvalues. On the other hand,  the curved Casimirs induce more operators due to the degeneracy. Namely, the operator $\Ca$ itself evidently gives an invariant operator between first two slots and an invariant operator between the last two slots. It is easy to see that these operators are nothing else but the exterior derivatives. The next operator we get is due to the coincidence of eigenvalues $\beta_1=\beta_3$. By the construction principle, an invariant operator between the respective bundles
$$
\M_{AB}:\ce^{A'}\otimes\yiia\ce_{\dbA\ddbB}[-k+2]\rightarrow
\ce^{A'}\otimes\yiia\ce_{\bA\dbB}[-k+1]
$$ is induced by $\Ca\circ(\Ca-4)\circ(\Ca+4)$. Similarly, the coincidences $\beta_0=\beta_3$ and $\beta_1=\beta_4$ show that $\Ca^2\circ(\Ca-4)\circ(\Ca+4)$ gives rise to invariant operators between the corresponding bundles. Obviously, these operators are given by the compositions $\M\circ  d $ and $ d \circ\M$. But, there are no such operators in the classification of invariant operators on flat structures, and so the compositions must vanish in that case. We will prove that $\M\circ  d $ and $ d \circ\M$ vanish identically even in the case of a general torsion-free structure with a nonzero curvature. It implies then that $\Ca^2\circ(\Ca-4)\circ(\Ca+4)$ actually induces an invariant operator from the top slot to the bottom slot, and we will show that this is the curved analogue of the nonstandard operator that we wanted to construct. Let us remark that this construction breaks down in the case that the torsion does not vanish, see Remark \ref{torsion} bellow.

\begin{prop}
\label{thm}
Let $M$ be a manifold endowed with Grassmannian or quaternionic structure. For each integer $k$ such that $2\leq k\leq n$ the action of  operator $\Ca^2\circ(\Ca-4)\circ(\Ca+4)$ on tractor bundle   $\mathcal{T}$ gives rise to a fourth order invariant operator 
$$
\square_{ABCD}: \yiikwide{k-2}\ce_{\ddbA\ddbB}[-k+2]\to\yiik{k}\ce_{\bA\bB}[-k]
$$
which coincides with the corresponding nonstandard operator on flat structures.
\end{prop}
\begin{proof}
First, we prove  that for each $k$ such that $0\leq k\leq n-2$ the operators $\M\circ  d $ and $ d \circ\M$ vanish identically. In the second step, we prove that the principal part of $\square_{ABCD}$ is given by a nonzero scalar multiple of the corresponding nonstandard operator $\widetilde{\square}_{ABCD}$, and thus the operators coincide on the category of locally flat structures.

The direct way how to prove vanishing of the two operators is to express the operator $\M$ in terms of the Weyl connection and Rho-tensor. This can be easily done in a way since the operator is obtained by acting with $\Ca\circ(\Ca-4)\circ(\Ca+4)$, and we have  formula \eqref{Ca_short} for $\Ca$ at disposal. Namely, denoting the projections to the left slot (corresponding to $\beta^1_2$) and right slot (corresponding to $\beta^2_2$) by $(\ph)_1$ and $(\ph)_2$ respectively, it is an easy exercise to show that up to a scalar multiple $\M$ can be written as
\begin{equation}
\label{M}
\M=\nabla \bullet (\nabla\bullet\ph)_1-\nabla\bullet(\nabla\bullet\ph)_2 +2\Pe\bullet\bullet\ph .
\end{equation} 
For more details see \cite[section 3.1.4]{thesis}.
Obviously, the compositions $\M\circ  d $ and $ d \circ\M$ are given by $\M (\nabla\bullet\ph)$ and $\nabla\bullet\M$ respectively. 
Having these abstract formulas in hand, we can directly use  the form of the action $\bullet$ given by equation \eqref{action}  to show explicitly that the operators vanish. The computation becomes tedious soon however. Therefore, we rather prove the vanishing by analyzing terms that might occur in formulae for these operators.

The first observation is that the principal (third order) part of operators $\M (\nabla\bullet\ph)$ and $\nabla\bullet\M$ vanishes in the torsion-free case, and thus $\M\circ  d $ and $ d \circ\M$ are operators of order at most one with the curvature in their leading part. This follows from the classification of invariant operators in the flat case -- it is well known that there are no operators between the respective bundles, see e.g. \cite{BoeCollingwood}. It can be also checked directly by use of \eqref{action} for the first two terms in  \eqref{M}. This is doable since  we may freely commute  the derivatives.   It can be also partially deduced from the computation of the principal part of $\square_{ABCD}$ below.

Since there are no terms of order three in $\M\circ  d (\varphi)$ and $ d \circ\M(\psi)$ for each $\varphi$ and $\psi$, the operators are given by projections to target bundles of a linear combination  of terms of the form $R\cdot\nabla\phi$ and $\nabla R\cdot\phi$, where $\phi=\varphi$ and $\phi=\psi$ respectively. According to the description of curvature in section \ref{GS}, the operators may be also written as projections of a linear combination of  terms of the form  $W\cdot\nabla\phi, \nabla W\cdot\phi$ and $\Pe\cdot\nabla\phi, \nabla \Pe\cdot\phi$, where the tensor $\Pe_{ab}$ is symmetric, and the symmetries of tensor $W_{ab}{}^d{}_{c}$  and exterior derivative of the Rho-tensor are described in \eqref{W_sym} and \eqref{Pe_sym}  respectively.
Since $\M\circ  d $ and $ d \circ\M$ are natural by construction, given a point $x$ we can compute $M( d \varphi)(x)$ and $ d (\M\psi)(x)$ for each $\varphi, \psi$ in terms of any distinguished connection. So let us choose a distinguished connection $\nabla_a$ such that $\Pe_{(ab)}=0$ and also $\nabla_{(a}\Pe_{bc)}=0$ at $x$. Its existence  follows from the transformation formulas of the Rho-tensor and its covariant derivative. One can actually demand that the total symmetrizations of all derivatives of $\Pe_{ab}$  vanish at $x$. This is a general feature which refers to the so called normal scale for a parabolic geometry, see \cite[theorem 5.1.12]{parabook}. Evidently,  $M( d \varphi)(x)$ and $ d (\M\psi)(x)$ for such a connection are then given by projections of terms which contain either tensor $W_{ab}{}^d{}_{c}$ or tensor $(d^\nabla\Pe)_{abc}=\nabla_{[a}\Pe_{b]c}$. However,  by \eqref{W_sym} and \eqref{Pe_sym} both tensors are  symmetric in three unprimed spinor indices while the Young diagrams of target bundles has only two columns - recall that the target bundles of the projections are the two bundles in the bottom of composition series \eqref{comp_series}. Thus all terms vanish under the projections. Since $x$, $\varphi$, $\psi$ were arbitrary, we conclude that the invariant operators $\M\circ  d $ and $ d \circ\M$ are the zero operators.

Now we know that $\square_{ABCD}$ is invariant. To finish the proof,  we need to show that its principal part is given by  a nonzero scalar multiple of $\widetilde{\square}_{ABCD}$. First, we find even a full formula in an abstract form. By our construction, $\square_{ABCD}$ is given by a formula in the bottom slot of the section of tractor bundle \eqref{comp_series} which is given by the action of $\Ca^2\circ(\Ca-4)\circ(\Ca+4)$ on the section with $\sigma$ in the top slot and zeros elsewhere.
Acting with $\Ca^2$ first, by formula \eqref{Ca_short} for $\Ca$ we get zeros everywhere except for the middle slot and the slot below, where up to the factor 4 we get  
$(\nabla\bullet\nabla\bullet\sigma)_{1,2}\mp 2(\Pe\bullet\bullet\sigma)_{1,2}$ and $\Pe\bullet\bullet\nabla\bullet\sigma$ respectively. Further, acting  with $(\Ca-4)\circ(\Ca+4)$ we get zeros everywhere except for the bottom slot (in the slot above the bottom there is $\M \operatorname({ d }\sigma)=0$), and it is easy to deduce that, up to a scalar multiple, the formula there reads as
\begin{equation}  
\label{Do}
\begin{array}{l}
\square \sigma=
\nabla\bullet\nabla\bullet (\nabla\bullet\nabla\bullet \sigma)_1+\nabla\bullet\nabla\bullet (\nabla\bullet\nabla\bullet \sigma)_2
-2\Pe\bullet\bullet (\nabla\bullet\nabla\bullet \sigma)_1
\\
+2\Pe\bullet\bullet(\nabla\bullet\nabla\bullet \sigma)_2
-2\nabla\bullet\nabla\bullet(\Pe\bullet\bullet \sigma)_1+2\nabla\bullet\nabla\bullet(\Pe\bullet\bullet \sigma)_2.
\end{array}
\end{equation}
A detailed derivation of this formula can be found in \cite[section 3.1.6]{thesis}. Thus we see that the principal part is given by the sum of the two possible paths from the top slot to the bottom slot.
Now we make the corresponding two terms  explicit by a multiple use of formula \eqref{action} for $\bullet$. We directly get  that the term corresponding to the path through slot $(\ph)_1$ is given by 
$
8\nabla_{J'A_1}\nabla_{I'B_1}\nabla^{(I'}_{(B_2}\nabla^{J')}_{A_2)}\sigma_{\ddbA\ddbB},
$
followed by the (unique) projection to the target bundle. Recall that the target bundle is the bundle in the bottom slot of \eqref{comp_series}, i.e.  \begin{equation}
\label{target}
\yii\ce_{\bA\bB}[-k]\subset\ce_{(\bA\bB)}[-k]
\end{equation}
and thus the projection is given by alternation in $A_1,\dots,A_k$ and in $B_1,\dots,B_k$, followed by symmetrization in $\bA$ and $\bB$ and projection to the joint kernel of alternations in more than $k$ inputs.
Note that we may freely commute the derivatives since we are only interested in the principal part. In particular, the symmetry of primed indices of covariant derivatives translates to unprimed indices and vice versa. Hence expanding the symmetrizations in the previous formula, we get an equivalent formula
$$
4\nabla_{J'A_1}\nabla^{J'}_{A_2}\nabla_{I'B_1}\nabla^{I'}_{B_2}\sigma_{\ddbA\ddbB}
+4\nabla_{J'A_1}\nabla^{I'}_{A_2}\nabla_{I'B_1}\nabla^{J'}_{B_2}\sigma_{\ddbA\ddbB}.
$$  
Moreover, since the explicit form of the isomorphism $\ce^{[A'B']}\cong\ce[-1]$ reads as $v^{[A'B']}=-1/2 \cdot v_{I'}{}^{I'}\ep^{A'B'}$, the second summand is the previous display  is equivalent to $-\nabla_{J'A_1}\nabla^{J'}_{A_2}\nabla_{I'B_1}\nabla^{I'}_{B_2}\sigma_{\ddbA\ddbB}.$ Therefore,  the principal part of $\nabla\bullet\nabla\bullet (\nabla\bullet\nabla\bullet \sigma)_1$ is given by the projection of
$
3\nabla_{J'A_1}\nabla^{J'}_{A_2}\nabla_{I'B_1}\nabla^{I'}_{B_2}\sigma_{\ddbA\ddbB},
$
and thus it coincides with $3\cdot\widetilde{\square}_{ABCD}$.
Similarly, we deduce directly by a multiple use of \eqref{action} that the leading term corresponding to the path through slot $(\ph)_2$ is given by 
\begin{equation}
\label{path_2}
4\nabla_{I'A_1}\nabla^{I'}_{A_2}\nabla_{J'B_1}\nabla^{J'}_{B_2}\sigma_{\ddbA\ddbB}
-2k(-1)^k\nabla_{I'A_1}\nabla^{I'}_{B_1}\nabla_{J'[B_2}\nabla^{J'}_{B_3}
\sigma_{\dddbB A_2] \ddbA},
\end{equation}
followed by the projection to target bundle \eqref{target}. 
The second term can be rewritten by expanding the displayed alternation  with respect to $A_2$ and $\ddbB$ as follows
$$
\begin{array}{l}
\nabla_{J'[B_2}\nabla^{J'}_{B_3}
\sigma_{\dddbB A_2] \ddbA}=
\frac{1}{k}(-1)^{k-1}\nabla_{J'A_2}\nabla^{J'}_{[B_2}\sigma_{\ddbB]\ddbA}
+\frac{1}{k}(-1)^{k-2}\nabla_{J'[B_2}\nabla^{J'}_{|A_2|}\sigma_{\ddbB]\ddbA}
\\
+\frac{k-2}{k}\nabla_{J'[B_2}\nabla^{J'}_{B_3}\sigma_{\dddbB]A_2\ddbA}.
\end{array}
$$
Since $\sigma_{\ddbA\ddbB}$ is a section of the bundle which corresponds to Young diagram with two columns of height $k-2$, the last term on the right hand side in the previous equation vanishes under the alternation over $A_2,\dots,A_k$. Therefore, we can substitute $\nabla_{J'[B_2}\nabla^{J'}_{B_3}
\sigma_{\dddbB A_2] \ddbA}$ in \eqref{path_2} by $\frac{2}{k}(-1)^{k-1}\nabla_{J'[A_2}\nabla^{J'}_{B_2]}\sigma_{\ddbA\ddbB}$, and thus we conclude that  the principal part of $\nabla\bullet\nabla\bullet (\nabla\bullet\nabla\bullet \sigma)_2$ is given by the projection of
$$
4\nabla_{I'A_1}\nabla^{I'}_{A_2}\nabla_{J'B_1}\nabla^{J'}_{B_2}\sigma_{\ddbA\ddbB}
+4\nabla_{I'[A_1}\nabla^{I'}_{B_1]}\nabla_{J'[A_2}\nabla^{J'}_{B_2]}\sigma_{\ddbA \ddbB}
$$
to the target bundle.
Now we observe that this is actually the same formula as in the case of the first path up to a  commutation of derivatives.  Hence by the same reasons as above, the principal part of  $\nabla\bullet\nabla\bullet (\nabla\bullet\nabla\bullet \sigma)_2$  coincides with $3\cdot\widetilde{\square}_{ABCD}$. The principal part of $\square_{ABCD}$ then equals  $6\cdot\widetilde{\square}_{ABCD}$.
\end{proof}

The proposition shows in particular that in the torsion-free case there exist a curved analogue for each of the nonstandard operators. Hence it gives an alternative proof of Theorem 5.1 of \cite{Gover}. Moreover, our construction directly yields a formula  for each of these curved analogues, cf. formula \eqref{Do}. Making this formula explicit and using the symmetries of the Rho-tensor and its exterior derivative, we get the following.
\begin{cor}
\label{factorization}
The  Grassmannian nonstandard operator can be written as 
the projection of an operator $ d \circ A\circ d $ to  bundle \eqref{target}, where $A$ is a noninvariant operator 
$$
A_{AB}:\ce^{A'}\otimes\yiia\ce_{\dbA\ddbB}[-k+2]\rightarrow
\ce^{A'}\otimes\yiia\ce_{\bA\dbB}[-k+1]
$$
which is given by 
$$
\begin{array}{l}
(A\mu)^{A'}_{\bA\dbB}=
2\nabla_{J'A_1}\nabla^{(A'}_{A_2}\mu^{J')}_{\dbB\ddbA}
+2\nabla_{J'A_1}\nabla^{(A'}_{B_2}\mu^{J')}_{\dbA\ddbB}
+2\nabla^{A'}_{B_2}\nabla_{J'A_1}\mu^{J'}_{\dbA\ddbB}
\\
\phantom{qqq}
+\nabla^{A'}_{A_1}\nabla_{J'B_2}\mu^{J'}_{\dbA\ddbB}
-\nabla^{A'}_{A_1}\nabla_{J'A_2}\mu^{J'}_{\dbB\ddbA}
+16\Pe^{(A'J')}_{(B_2A_1)}\mu_{J'\dbA\ddbB}
+8\Pe_{J'A_1A_2}^{\ph\ph J'}\mu^{A'}_{\dbB\ddbA}
\\
\phantom{qqq}
+8\Pe_{J'[A_1B_2]}^{\ph\ph J'}\mu^{A'}_{\dbA\ddbB}.
\end{array}
$$
\end{cor}
\begin{proof}
Applying the explicit form of the action $\bullet$ from \eqref{action} to formula \eqref{Do}, we directly get that $(\square f)_{\bA\bB}$ is given by the action of
$$
\begin{array}{l}
\nabla_{A_1A_2B_1B_2}=
8\nabla_{I'(A_1}\nabla_{|J'|B_1)}\nabla^{(I'}_{(A_2}\nabla^{J')}_{B_2)}
+2\nabla_{I'A_1}\nabla^{I'}_{A_2}\nabla_{J'B_1}\nabla^{J'}_{B_2}
\\
\phantom{qqq}
+2\nabla_{I'B_1}\nabla^{I'}_{B_2}\nabla_{J'A_1}\nabla^{J'}_{A_2}
+4\nabla_{I'[A_1}\nabla^{I'}_{B_1]}\nabla_{J'[A_2}\nabla^{J'}_{B_2]}
\\
\phantom{qqq}
-16\Pe_{I'(A_1|J'|B_1)}\nabla^{(I'}_{(A_2}\nabla^{J')}_{B_2)}
+4\Pe_{I'\ph A_1A_2}^{\ph\ph I'}\nabla_{J'B_1}\nabla^{J'}_{B_2}
\\
\phantom{qqq}
+4\Pe_{I'\ph B_1B_2}^{\ph\ph I'}\nabla_{J'A_1}\nabla^{J'}_{A_2}
+8\Pe_{I'\ph [A_1B_1]}^{\ph\ph I'}\nabla_{J'[A_2}\nabla^{J'}_{B_2]}
\\
\phantom{qqq}
-16\nabla_{I'(A_1}\nabla_{|J'|B_1)}\Pe^{(I'J')}_{(A_2B_2)}
+4\nabla_{I'A_1}\nabla^{I'}_{A_2}\Pe_{J'\ph B_1B_2}^{\ph\ph J'}
\\
\phantom{qqq}
+4\nabla_{I'B_1}\nabla^{I'}_{B_2}\Pe_{J'\ph A_1A_2}^{\ph\ph J'}
+8\nabla_{I'[A_1}\nabla^{I'}_{B_1]}\Pe_{J'[A_2B_2]}^{\ph\ph J'},
\end{array}
$$
on a section $f_{\ddbA\ddbB}$. Now we simplify this formula by using symmetries \eqref{Pe_sym} of the covariant derivative of the Rho-tensor. Namely, we can replace each tensor $\nabla\Pe$ by its totally symmetric part since all other components are symmetric in three unprimed spinor indices and thus all terms containing them vanish when projected to target \eqref{target}. Hence the we get the identity
$$
-2\nabla_{J'B_1}\Pe^{(I'J')}_{(A_2B_2)}
+\nabla^{I'}_{A_2}\Pe^{\ph\ph J'}_{J'B_1B_2}
+\nabla^{I'}_{B_1}\Pe^{\ph\ph J'}_{J'[A_2B_2]}=0
$$
Applying this equation to the previous formula we conclude that the sum of terms of type $\nabla\nabla\Pe f$ is equal to the sum  of terms of type $\Pe\nabla\nabla f$ and that the lower order terms may be written as 
$$
\begin{array}{l}
\nabla_{I'A_1}(-16\Pe^{(I'J')}_{(A_2B_2)}\nabla_{J'B_1}
+8\Pe^{\ph\ph J'}_{J'B_1B_2}\nabla^{I'}_{A_2}+8\Pe^{\ph\ph J'}_{J'[A_2B_2]}\nabla^{I'}_{B_1})f_{\ddbA\ddbB}
\\
+\nabla_{I'B_1}(-16\Pe^{(I'J')}_{(B_2A_2)}\nabla_{J'A_1}
+8\Pe^{\ph\ph J'}_{J'A_1A_2}\nabla^{I'}_{B_2}+8\Pe^{\ph\ph J'}_{J'[B_2A_2]}\nabla^{I'}_{A_1})f_{\ddbA\ddbB}.
\end{array}
$$
Rewriting the leading part accordingly, the result follows by applying formulas for the exterior derivative
 $(df)^{A'}_{\dbA\ddbB}=\nabla^{A'}_{[A_2}f_{\ddbA]\ddbB}$ and $(d\nu)_{\bA\bB}=\nabla_{I'[A_1}\nu^{I'}_{\dbA]\bB}
+\nabla_{I'[B_1}\nu^{I'}_{\dbB]\bA}$.
 \end{proof}

\begin{rem}
\label{torsion}
{\upshape
The torsion-freeness of the structure is important for the proof of proposition \ref{thm}. In the case of nonvanishing torsion, the Weyl curvature does not lie in the irreducible bundle \eqref{W_sym} but consists of the harmonic part and some other irreducible components which can be expressed in terms of torsion. Of course, also the derivative of the Rho-tensor does not lie in \eqref{Pe_sym}. The consequence of these facts is that the operators $\M\circ  d $ and $ d \circ \M$ do not vanish. Namely, it is easy to show that they both are second order operators with the torsion in their leading part. Our construction breaks down therefore. However, this does not mean that curved analogues of the nonstandard operators do not exist in such a case. Indeed, it is proved in \cite[Section 3.2]{thesis} that the operator which acts on functions exists also in the presence of a nonzero torsion. On the other hand, it is proved there that this operator cannot be written as in corollary \ref{factorization}, i.e. as a composition of operators with the exterior derivatives in the beginning and at the end. 
} 
\end{rem}

\section{Weak invariance of nonstandard operators}

The invariance of operators $\square_{ABCD}$ constructed above obviously depends on the vanishing of operators $\M\circ  d $ and $ d \circ\M$ in the torsion-free case. The first step of the proof of the vanishing of these operators was an observation that  their leading (third order) terms  can be rewritten in terms of the first derivative and the curvature.  Hence we need to commute covariant derivatives in order to prove the invariance of operators $\square_{ABCD}$. This shows in turn that if we replace the distinguished connection $\nabla$ by the coupled distinguished tractor connection, defined by the Leibnitz rule in the usual way, then  the formula obtained from the construction does not define invariant operator on tractor bundles. The transformation of such operator will consist of terms containing tractor curvature in general. A natural question now is whether there exist a formula for $\square_{ABCD}$ which is universal in the sense that it  defines invariant operators also on tractor bundles or there is no such formula, i.e. whether $\square_{ABCD}$ is strongly invariant or not.  We will mainly use a slightly different notion of strong invariance. Namely, an operator is strongly invariant in the sense of  \cite{Eastwood} if it factors through semi-holonomic jets. In the dual picture, it means that it is induced by  a homomorphism of semi-holonomic Verma modules. Let us remark that such an operator then translates to  tractor bundles, see \cite{translation}. We will prove in this section that operators  $\square_{ABCD}$ are not strongly  invariant in the algebraic sense. In the subsequent remark, we will  argue that the operators are also not strongly invariant in the sense of the existence of a universal formula. 

\subsection{Nonstandard operators are not strongly invariant}
Let us digress to the case of locally flat structures for a moment. It is well known that then  the jet bundles are associated to the Cartan bundle. Hence invariant differential operators are in a bijective correspondence with homomorphisms between jet prolongations of  representations which induce the bundles in question. According to the definition of the flat nonstandard operators $\widetilde{\square}_{ABCD}$, the respective homomorphisms  are $\mathfrak{p}$-homomorphisms $\Phi: \mathcal{J}^4(\mathbb{V}_{k-2})\to\mathbb{V}_{k}$, where for each $2\leq k\leq n$ we set
$$
\mathbb{V}_k:=\yiik{k}\R^{n*}[-k]
$$
for the module inducing bundle \eqref{target}, and  where $\mathcal{J}^4(\mathbb{V}_{k})$ denotes  a module inducing the fourth  jet prolongation of this bundle. By the description of  the flat nonstandard operators  in section \ref{NO},  for each convenient $k$ the map $\Phi$ is given by a composition of  (up to a multiple) unique  $\g_0$-homomorphism 
$\phi: S^4\g_{-1}^*\otimes\mathbb{V}_{k-2}
\to\mathbb{V}_{2}\otimes\mathbb{V}_{k-2}$, which in terms of abstract indices reads
\begin{equation}
\label{Phi}
\phi(\om)_{ABCD\bE\bF}=
\frac12(\om_{I'[AB]J'[CD]\bE\bF}^{\ph\ph I' \ph\ph\ph\ph\ph\ph J'}
+\om_{I'[CD]J'[AB]\bE\bF}^{\ph\ph I' \ph\ph\ph\ph\ph\ph J'})
-\om_{I'[AB|J'|CD]\bE\bF}^{\ph\ph I' \ph\ph\ph\ph\ph\ph J'},
\end{equation}
with a unique projection $\mathbb{V}_{2}\otimes\mathbb{V}_{k-2}\to\mathbb{V}_{k}$. We will manifest in the course of the forthcoming proof that $\Phi$ is indeed a $\mathfrak{p}$-homomorphism.

In contrary to ordinary jet bundles, the  semi-holonomic jet bundles are associated  to the Cartan bundle also  on structures with nonzero curvature. The respective representation, denoted by $\bar{\mathcal{J}}^k()$, is called the semi-holonomic jet prolongation. The dual module is the  so called semi-holonomic Verma module. For more details see  \cite{CSS1} and  \cite{Eastwood}. The operators which are strongly invariant in the algebraic sense are exactly the operators which are induced by homomorphisms between these modules. We will prove that  operator $\square_{ABCD}$ does not belong to them. That is,  there does not exist  any $\mathfrak{p}$-homomorphism $\widetilde{\Phi}: \bar{\mathcal{J}}^4(\mathbb{V}_{k-2})\to \mathbb{V}_k$ such that its restriction to the holonomic jets  $ \mathcal{J}^4(\mathbb{V}_{k-2})\subset \bar{\mathcal{J}}^4(\mathbb{V}_{k-2})$ coincides  with $\Phi$ (this condition says in other words that $\square_{ABCD}$ coincides with $\widetilde{\square}_{ABCD}$ on locally flat structures).
\begin{prop}
\label{weak_invariance}
The Grassmannian nonstandard operators  are not strongly invariant.
\end{prop}
\begin{proof}
We shall prove that there exists no $\mathfrak{p}$-homomorphism  $\widetilde{\Phi}$ which extends $\Phi$. We prove it in two steps. 
By \cite[Lemma 5.8]{CSS1}, 
a $\g_0$-homomorphism  $\widetilde{\Phi}$ is $\mathfrak{p}$-homomorphism if and only if it factors through $\mathcal{J}^4(\mathbb{V}_{k-2})\to\otimes^{4}\g_{-1}^*\otimes \mathbb{V}_{k-2}$ and it vanishes on the image of $\otimes^3\g_{-1}^*\otimes\mathbb{V}_{k-2}$ under the action of $\g_1$.
Hence we analyze  first the space of $\g_0$-homomorphisms $\widetilde{\Phi}: \otimes^4\g_{-1}^*\otimes\mathbb{V}_{k-2}\to\mathbb{V}_{k}$, and then we 
describe explicitly the image of the action of $\g_1$ on $\bar{\mathcal{J}}^4(\mathbb{V}_{k-2})$ and we show that it never lies in  $\operatorname{Ker}(\widetilde{\Phi})$ for any $\widetilde{\Phi}$ which lifts $\Phi$.

Step 1. 
An important observation is that any $\g_0$-homomorphism  $\otimes^4\g_{-1}^*\otimes\mathbb{V}_{k-2}\to\mathbb{V}_{k}$ factors through $\mathbb{V}_{2}\otimes\mathbb{V}_{k-2}$ and  a unique projection $\mathbb{V}_{2}\otimes\mathbb{V}_{k-2}\to\mathbb{V}_{k}$. 
Therefore, we only need to analyze $\g_0$-homomorphisms  $\otimes^4\g_{-1}^*\to\mathbb{V}_{2}$. This map is a complete contraction on the $\R^2$-part of $\g_{-1}^*$ and thus a linear combination of $c_1(\om)=\om_{I'}{}^{I'}{}_{J'}{}^{J'}$, $c_2(\om)=\om_{I'J'}{}^{J'I'}$ and $c_3(\om)=\om_{I'}{}^{J'I'}{}_{J'}$, which are related by $c_1+c_2+c_3=0.$ 
Similarly, on the $\R^{n*}$-part of $\g_{-1}^*$ we denote the three projections  $\otimes^4\R^{n*}\to\yiiII\R^{n*}$  as follows
$$
p_1(\om)_{\bA^2\bB^2}=\tfrac12(\om_{A_1A_2B_1B_2}+\om_{B_1B_2A_1A_2})
-\om_{[A_1A_2B_1B_2]},
$$
$$
p_2(\om)_{\bA^2\bB^2}=\tfrac12(\om_{A_1B_1B_2A_2}+\om_{B_1A_1A_2B_2})
-\om_{[A_1A_2B_1B_2]},
$$
$$
p_3(\om)_{\bA^2\bB^2}=\tfrac12(\om_{A_1B_2A_2B_1}+\om_{B_1A_2B_2A_1})
-\om_{[A_1A_2B_1B_2]}.
$$
They are obviously related by a similar equation $p_1+p_2+p_3=0.$ Hence the $\g_0$-homomorphisms  $\otimes^4\g_{-1}^*\to\mathbb{V}_{2}$ form a vector space of dimension four. We choose $c_i\circ p_j$, $i,j=1,2$ as a basis and
we set $\Phi_{ii}:=(c_i\circ p_i)\otimes\id$ and $\Phi_{ij}:=-2(c_i\circ p_j)\otimes\id$ for $i\neq j$. Then any $\g_0$-homomorphism  $\widetilde{\Phi}: \otimes^4\g_{-1}^*\otimes\mathbb{V}_{k-2}\to\mathbb{V}_{k}$ is given by a linear combination 
$K\Phi_{11}+L\Phi_{12}+M\Phi_{21}+N\Phi_{22}$, followed by the unique projection to $\mathbb{V}_{k}$. Moreover, the uniqueness in symmetric case implies that the restriction of each $\Phi_{ij}$ to $S^4\g_{-1}^*\otimes\mathbb{V}_{k-2}$  is a multiple of $\phi$, which is given by \eqref{Phi} and which defines the holonomic map $\Phi$. Precisely, it it is easy to compute that upon the restriction to $\mathcal{J}^4(\mathbb{V}_{k-2})$ all maps $\Phi_{ij}$ coincide with $\phi$.
This shows that $\widetilde{\Phi}$ covers $\Phi$ if and only if  it is given by a linear combination of maps $\Phi_{ij}$ such that the coefficients satisfy $K+L+ M+N=1.$

Step 2. 
By \cite[Lemma 5.10]{CSS1},  the action of $Z\in \g_1$ on $\psi\in\otimes^{3}\g_{-1}^*\otimes \mathbb{V}_{k-2}$, regarded as a $\g_0$-submodule of $\bar{\mathcal{J}}^4(\mathbb{V}_{k-2})$, is given by 
\begin{equation}
\label{Verma_action}
\begin{array}{l}
(Z\cdot\psi)(X_1,\cdots,X_4)=
\sum\limits_{1\leq i\leq 4} [Z,X_i]\bullet \psi(X_1,\cdots,\hat{X}_i,\cdots,X_{4})
\\
-\sum\limits_{1\leq i<j\leq 4}\psi(X_1,\cdots,[[Z,X_i],X_j]\hat{X}_j,\cdots,  X_{4})
\end{array}
\end{equation}
In terms of abstract indices with conventions from section \ref{GS}, the element $[Z,X_i]$ of $\g_0$ is given by $[Z,X_i]=(Z^{A'}_I(X_i)^I_{B'},Z^{I'}_{A}(X_i)^{B}_{I'})$, and its action on $v^{A'}\in\R^2$, $v_A\in\R^{q*}$, $\sigma\in\R[w]$ and its adjoint action on $X_j\in\g_{-1}$reads as follows
$$
\begin{array}{c}
[Z,X_i]\bullet v^{A'} = Z^{A'}_I(X_i)^I_{I'} v^{I'},
\\

[Z,X_i]\bullet v_A = Z^{I'}_{A}(X_i)^{I}_{I'} v_I,
\\

[Z,X_i]\bullet\sigma =wZ^{I'}_{I}(X_i)^{I}_{I'} \sigma,
\end{array}
$$
and
$$
[[Z,X_i],X_j]^{A}_{A'}=-Z^{I'}_{I}(X_i)^{A}_{I'}(X_j)^{I}_{A'}-Z^{I'}_I(X_i)^I_{A'}(X_j)^A_{I'}.
$$
Now we apply these equations to  \eqref{Verma_action}. It is easy to see that for  $Z_{a_i}=Z^{A'_i}_{A_i}\in \g_{-1}^*$ and $\psi_{a_1a_2a_3\bE\bF}\in\otimes^3\g_{-1}^*\otimes\mathbb{V}_{k-2}$ the terms $[Z,X_i]\bullet \psi$ appearing in the first sum are given by
$$
\begin{array}{l}
( [Z,X_i]\bullet \psi)_{a_1a_2a_3a_4\bE\bF}
=(k-2)Z^{A'_i}_{A_i}\psi_{a_1\cdots\hat{a}_i\cdots a_4  \bE\bF}
\\
+(k-2)Z^{A'_i}_{E_1}\psi_{a_1\cdots\hat{a}_i\cdots a_4  A_i\dbE\bF}
+(k-2)Z^{A'_i}_{F_1}\psi_{a_1\cdots\hat{a}_i\cdots a_4  \bE A_i\dbF},
\end{array}
$$
Now it is an easy observation that the unprimed spinor indices displayed on the right-hand side of  the previous formula lie in
$$
\yiiIkwide{k-2}\R^{n*}.
$$
It means that each term $[Z,X_i]\bullet \psi$ vanishes under the projection to $\mathbb{V}_{k}$, and thus only the second sum in \eqref{Verma_action} remains modulo Ker($\widetilde{\Phi}$). For $\psi=\om\otimes v$, where $\om\in\otimes^3\g_{-1}^*$ and $v\in\mathbb{V}_{k-2}
$, it reads as follows
$$
\begin{array}{l}
(Z\cdot(\om\otimes v))^{A'B'C'D'}_{ABCD}=(Z^{A'}_B\om^{B'C'D'}_{ACD}+Z^{B'}_A\om^{A'C'D'}_{BCD}
+Z^{A'}_C\om^{B'C'D'}_{BAD}+Z^{C'}_A\om^{B'A'D'}_{BCD}
\\
+Z^{A'}_D\om^{B'C'D'}_{BCA}+Z^{D'}_A\om^{B'C'A'}_{BCD}+Z^{B'}_C\om^{A'C'D'}_{ABD}
+Z^{C'}_B\om^{A'B'D'}_{ACD}+Z^{B'}_D\om^{A'C'D'}_{ACB}
\\
+Z^{D'}_B\om^{A'C'B'}_{ACD}+Z^{C'}_D\om^{A'B'D'}_{ABC}+Z^{D'}_C\om^{A'B'C'}_{ABD})\otimes v.
\end{array}
$$
In order to show that for some $\om$ it does not lie in Ker($\widetilde{\Phi}$) for any lift $\widetilde{\Phi}$, we express its image under each $\Phi_{ij}$.
A straightforward computation yields
$$
\begin{array}{l}
c_1(Z\cdot\om)_{ABCD}=
-Z_{I'C}(\om^{I'J'}_{BADJ'}+\om_{J'ABD}^{\ph\ph\ph I'J'})
\\
+Z_{I'D}(-\om^{I'J'}_{BCAJ'}+\om^{I'J'}_{ACBJ'}+\om_{J'ABC}^{\ph\ph\ph J'I'})
\end{array}
$$
and 
$$
\begin{array}{l}
c_2(Z\cdot\om)_{ABCD}=
Z_{I'C}(\om_{J'BAD}^{\ph\ph\ph J'I'}+\om_{J'ABD}^{\ph\ph\ph J'I'})
\\
+Z_{I'D}(\om_{J'BCA}^{\ph\ph\ph J'I'}+\om_{J'ACB}^{\ph\ph\ph I'J'}-\om_{J'ABC}^{\ph\ph\ph I'J'}).
\end{array}
$$
Now it is easy to see that both contractions vanish provided that  $\om$ is symmetric. And since holonomic  homomorphism $\Phi$ factors through a  complete contraction, it also vanishes on the image of the action of $\g_1$ on $S^3\g_{-1}^*\otimes\mathbb{V}_{k-2}$. This shows that $\Phi$ is indeed a $\mathfrak{p}$-homomorphism.  
On the other hand, the contractions $c_1$, $c_2$ for a nonsymmetric $\om$ are nonzero in general. Namely, if we set 
$$
\om^{A'B'C'}_{ABC}=\ep^{A'B'}\bar{\om}^{C'}_{ABC}\in\operatorname{Ker}(\ce^{C'}_{[AB]C}[-1]\to\ce^{C'}_{[ABC]}[-1])\subset\otimes^3\g_{-1}^*,
$$
then  they are equal to 
$$
c_1(Z\cdot\bar{\om})_{ABCD}=-3Z_{I'D}\bar{\om}^{I'}_{ABC} 
$$
$$
c_2(Z\cdot\bar{\om})_{ABCD}=-3Z_{I'D}\bar{\om}^{I'}_{BCA}.
$$
Now it is easy to  compute that the  compositions with projections $p_1,p_2$ yield
$$
p_1\circ c_1(Z\cdot\bar{\om})=-2 p_2\circ c_1(Z\cdot\bar{\om})
=-2p_1\circ c_2(Z\cdot\bar{\om})= p_2\circ c_2(Z\cdot\bar{\om}).
$$
This fact shows that  for  $\bar{\psi}:=\bar{\om}\otimes v$  the element $Z\cdot\bar{\psi}$ has the same image under all maps $\Phi_{ij}$. Hence the image of $Z\cdot\bar{\psi}$ under a lift $\widetilde{\Phi}$  of $\Phi$, which is given by a linear combination of maps $\Phi_{ij}$  with coefficients $K,L,M,N$,   is given by the projection of 
$$
(K+ L+ M+N)\Phi_{11}(Z\cdot\bar{\psi})=\Phi_{11}(Z\cdot\bar{\psi})
$$
to $\mathbb{V}_{k}$. In particular, it does not depend on the coefficients $K,L,M,N$ and it is not zero. Precisely, we get
$$
\widetilde{\Phi}(Z\cdot\bar{\psi})_{\bA\bB}=
\frac12(
Z_{I'B_2}\bar{\psi}^{I'}_{A_1A_2B_1\ddbA\ddbB}
+Z_{I'A_2}\bar{\psi}^{I'}_{B_1B_2A_1\ddbA\ddbB}
)
\neq 0.
$$
\end{proof}

\begin{rem}
{\upshape
A careful computation reveals that the maps $\Phi_{ij}$ coincide on whole image of $\otimes^3\g_{-1}^*\otimes\mathbb{V}_{k-2}$  in  $\otimes^4\g_{-1}^*\otimes\mathbb{V}_{k-2}$ under the action of $\g_1$. Since the action of $\g_1$ gives exactly the  transformation of the four-fold covariant derivative which is algebraic and linear in the one-form $\Upsilon$ describing the change $\nabla\to\hat{\nabla}$,  the coincidence of  maps $\Phi_{ij}$ shows in turn that the projection of the algebraic linearized transformation of  $\nabla^4f$ to $\mathbb{V}_{k}$ does not depend on the succession of  covariant derivatives. Indeed, acting with $\square_{ABCD}$ on a tractor bundle, the leading  terms with different orders of derivatives differ by terms of the form $\nabla^2\Omega$, $\nabla\Omega\nabla$ and $\Omega\nabla^2$ where $\Omega$ is the tractor curvature. And it is easy to prove that  projections to the target bundle of all these terms  are invariant. Hence any formula for $\square_{ABCD}$ will have the same linearized transformation of the leading part as the formula obtained from the curved Casimir construction. If we trace where we commuted derivatives back in the proof of proposition \ref{thm}, we conclude that this transformation (described by an one form $\Upsilon$)  is given by the projection of 
$$
\Upsilon\nabla\Omega+\Upsilon\Omega\nabla.
$$
Moreover, it is easy to see that this cannot be cancelled by transformations of lower order terms and thus it is contained in the transformation of any formula for $\square_{ABCD}$.  A straightforward computation shows that the action of this curvature expression on the standard tractor bundle vanishes if and only if the harmonic part of the Cartan curvature vanishes, i.e. in the flat case. Hence there exists no universal formula for  $\square_{ABCD}$ which would define invariant operator between tractor bundles on manifolds with a general Grassmannian structure.
}
\end{rem}


\begin{thebibliography}{99}
 

\bibitem{parabook} A. \v{C}ap and J. Slov\'ak, 
Parabolic Geometries I: Background and General Theory,
Math. Surv. and Monographs, {\bf 154},
Amer. Math. Soc., Providence, RI, 2009.

\bibitem{CCCBGG} A. \v Cap, V. Sou\v cek, Curved Casimir operators
  and the BGG machinery, SIGMA Symmetry Integrability Geom. Methods
  Appl. \textbf{3} (2007) 111, 17~pp.

\bibitem{CCCex} A. \v Cap,  A.R. Gover and V. Sou\v cek, Conformally Invariant Operators via Curved Casimirs:
Examples, Pure Appl. Math. Q., to appear, available at arXiv:0808.1978.

\bibitem{prolongation} M. Hammerl, P. Somberg, V. Sou\v cek, J. \v Silhan, 
Invariant prolongation of overdetermined PDE's in projective, conformal and Grassmannian geometry,
Annals of Global Analysis and Geometry, Springer, 2012, 0232-704X.

\bibitem{Gover} Gover A.R., Slov\' ak J., {\it Invariant local twistor calculus for quaternionic structures and related geometries,} J. Geom. Phys. 32, No.1 (1999) 14-56


\bibitem{Lepowsky} J. Lepowsky, A generalization of the
Bernstein--Gelfand--Gelfand resolution, J. of Algebra 49 (1977)
496--511 

\bibitem{Boe}B. Boe. Homomorphisms between generalized Verma modules. {\it Trans. Amer. Math. Soc.} 356 (1): 159-184, (2004).

\bibitem{Navrat} N\'avrat A.,  Nonstandard invariant operators on
quaternionic geometries, MSc Thesis, Masaryk University in Brno, 2004

\bibitem{thesis} N\'avrat A.,  Nonstandard operators in
almost Grassmannian geometry, PhD Thesis,  University of Vienna, 2012

\bibitem{BoeCollingwood} Boe, Brian D.; Collingwood, David H.  Multiplicity free categories of highest weight representations. I, II. Commun. Algebra 18, No.4, 947-1032, 1033-1070 (1990)

\bibitem{Subcomplexes} A. \v Cap, V. Sou\v cek, Subcomplexes in curved BGG sequences, 	Ann. Math. (2012) arXiv:math/0508534

\bibitem{Eastwood} Eastwood M., Slov\' ak J., {\it Semi-holonomic Verma modules}, J. of Algebra, 197 (1997), 424-448

\bibitem{CSS1} A. \v Cap, J. Slov\'ak, V. Sou\v cek, Invariant 
operators on manifolds with almost Hermitian symmetric structures, I. 
Invariant differentiation, Acta Math. Univ. Commenianae, {\bf 66\/}
(1997), 33--69, electronically available at www.emis.de

\bibitem{translation} A. \v Cap, {\it Translation of natural operators on manifolds with AHS-structures}, Archivum Mathematicum (Brno), Tomus 32 (1996), 249-266. 


\end{thebibliography}
\end{document}